\documentclass[12pt,reqno]{amsart}
\usepackage{amsfonts,amssymb,amsmath}
\usepackage{enumitem}
\usepackage{ytableau}
\usepackage{datetime}
\usepackage{tikz}
\usepackage{xcolor}

\definecolor{fgreen}{RGB}{44,144, 14}

\renewenvironment{proof}{{\bfseries Proof.}}{\qed}

\numberwithin{equation}{section} 
\newtheorem{theorem}{Theorem}[section] 
\newtheorem{proposition}[theorem]{Proposition} 
\newtheorem{corollary}[theorem]{Corollary} 
\newtheorem{lemma}[theorem]{Lemma}

\theoremstyle{definition}
\newtheorem{definition}[theorem]{Definition} 
\newtheorem{remark}[theorem]{Remark} 
\newtheorem{example}[theorem]{Example}
\usepackage[colorlinks=true,citecolor=cyan,  backref, pagebackref,   urlcolor=blue,linkcolor=blue]{hyperref}
\newcommand*{\bigchi}{\mbox{\Large$\chi$}}
\def\D{\mathbb {D}}
\def\R{\mathbb {R}}
\def\C{\mathbb {C}}
\def\N{\mathbb {N}}
\def\H{\mathbb {H}}

\def\O{\mathbb {O}}
\def\E{\mathbb {E}}

\def\d{\mathbf{ d}}

\def\ZC{\mathcal {Z}}
\def\CC{\mathcal {C}}

\def\OC{\mathcal {O}}
\def\BC{\mathcal {B}}

\def\g{\mathfrak {g}}
\def\h{\mathfrak {h}}
\def\x{X_n}
\def\p{\mathfrak {p}}

\def\s{\mathfrak {s}}
\def\a{\mathfrak {a}}

\def\o{\mathfrak {o}}
\def\l{\mathfrak {l}}
\def\z{\mathfrak {z}}

\def\cds{\cdots}
\def\th{\theta}

\def\<>{\langle \cdot\, , \cdot \rangle}
\def\lto{\longrightarrow}

\begin{document} 
 \title[Adjoint reality ]{A note on adjoint reality  in simple complex Lie algebras}
 \author[K. Gongopadhyay  \and C. Maity]{Krishnendu Gongopadhyay \and Chandan Maity
 }
\thanks{ORCID ID: K. Gongopadhyay (0000-0003-4327-0660);   C. Maity (0000-0003-0673-3953)}
\address{Indian Institute of Science Education and Research (IISER) Mohali,
 Knowledge City,  Sector 81, S.A.S. Nagar 140306, Punjab, India}
\email{krishnendug@gmail.com, krishnendu@iisermohali.ac.in}
\address{Indian Institute of Science Education and Research (IISER) Berhampur,  Odisha, 
 India}
\email{cmaity@iiserbpr.ac.in}
\makeatletter
\@namedef{subjclassname@2020}{\textup{2020} Mathematics Subject Classification}
\makeatother

\begin{abstract}
Let $G$ be a Lie group with Lie algebra $\g$.   In \cite{GM}, an infinitesimal version of the notion of classical reality, namely adjoint reality, has been introduced. An element $X \in \g$ is adjoint real if $-X$ belongs to the adjoint orbit of $X$ in $\g$.  In this paper, we investigate the adjoint real and the strongly adjoint real semisimple elements in complex simple classical Lie algebras. We also prove that every element in a  complex symplectic Lie algebra is adjoint real.
\end{abstract}

\maketitle

\section{Introduction} 
In the group theoretical  set-up an element $g$ in a group $G$ is called \emph{real} or \emph{reversible} if it is conjugate to $g^{-1}$ in $G$.  An element $g$ is \emph{strongly real} or \emph{strongly reversible}  in $G$ if it is conjugate to $g^{-1}$ by an involution. Classification of real and strongly real elements in a group is a  problem of wide interest; see \cite{OS}, \cite{ST}.

Let $G$ be a Lie group with Lie algebra $\g$. Consider the natural {\rm Ad}$(G)$-representation of $ G $ on its Lie algebra $ \g $
$$  
{\rm Ad} \colon G \lto {\rm GL}(\g)\,.
$$
For $X\in \g$, {\it the adjoint orbit}  of $X$ in $\g$ is defined as $\OC_X:= \{{\rm Ad}(g)X\mid g \in G \}$.
If $X\in \g$ is semisimple, then $\OC_X$ is called {\it semisimple orbit}.  Understanding the adjoint orbits in a semisimple Lie group has been an intense area of research,  cf. \cite{CoMc}, \cite{Mc}.   For various results related to semisimple orbits, see \cite[Chapter 2]{CoMc}.

In \cite{GM}, the authors  introduced the notion of {\it  adjoint reality} which we recall now.   Consider the  Ad$(G)$-representation of $G$ on $\g$. For a linear Lie group $ G $, $ {\rm Ad}(g)X=gXg^{-1} $.
\begin{definition}[{\cite[Definition 1.1]{GM}}]\label{def1} 
An element $X\in \g$ is called {\it {\rm Ad}$_G$-real } if $-X \in \OC_X$. An  {\rm Ad}$_G$-real  element is called {\it strongly {\rm Ad}$_G $-real}  if $-X = {\rm Ad}(\tau) X $ for some $\tau\in G$ so that $\tau ^2 = {\rm Id}$.
\end{definition}  

This is an infinitesimal analogue of the reality in Lie groups.  
It was shown in \cite{GM} that  the reality of the unipotent elements in a Lie group and  the ${\rm Ad}_G$-reality of the nilpotent elements in the corresponding Lie algebras  are equivalent via the exponential map.  This correspondence was used to classify unipotent real elements in classical Lie groups in \cite{GM}.   However, this correspondence   does not necessarily hold in general. Nevertheless, classifying the adjoint reality in Lie agebras is a problem of independent algebraic interest, and it also helps to understand the real and strongly real elements in the image of the exponential map, thus providing understanding of reality in the Lie group to a large extent.  The above notion of adjoint reality has turned useful  in  classifying the strongly real elements in $ {\rm GL}_n(\H) $, $ {\rm GL}_n(\D) \ltimes \D^n$ for $ \D=\R,  ~\C $ or $ \H $,  respectively, cf.  \cite{GLM1},  \cite{GLM2}.

With this motivation,  it is a natural problem to investigate the  adjoint reality for semisimple orbits.     The aim of this note is   to  classify the  adjoint real and strongly adjoint real  semisimple elements  in the complex simple classical Lie algebras; see Theorem \ref{Pro-SS-sl-n-c},  Theorem \ref{thm-sonC-semisimple}, Theorem \ref{thm-sp-nC}.   The adjoint real nilpotent elements in these Lie algebras are  classified in \cite{GM}.  

 By the Jordan decomposition, every element in a semisimple Lie algebra decomposes as a unique sum of a semisimple and a unipotent element. 
 Thus, classifying an arbitrary adjoint real element in a Lie algebra is intimately related to such classification of  semisimple and  nilpotent elements.   We demonstrate this for the symplectic Lie algebras, i.e. semisimple Lie algebras of type $C_n$.   Recall that  symplectic group plays an vital role in many branch of Mathematics. Thus the characterisation of adjoint real elements in the symplectic Lie algebra  are fundamentally important which is done in Theorem \ref{spnc-real} by using description of the centralizers.
We note here that classifying strongly real elements using this idea might require further  technicalities as we have seen for the type $A_n$ Lie algebras in \cite{GLM3}.  Other than Lie algebras of type $A_n$ and $C_n$, adjoint reality for arbitrary elements in other semisimple Lie algebras are yet to be fully understood.

Given $X\in \g$ one  {       defines} the following subsets of $G$. The {\it  centralizer} and the {\it reverser } of an element $X$ in $ G $ are respectively defined as 
$${Z}_G(X): = \{ s \in G \mid sXs^{-1} = X \}, \hbox{ and \ } {R}_G(X) := \{ r \in G \mid rXr^{-1} = -X \}.$$
Note that $ {Z}_G(X)$ is a subgroup but the set ${R}_G(X)$  is a right coset of the centralizer ${Z}_G(X)$.  Thus the { reversing symmetry group} or the  { extended centralizer} ${E}_G(X) := {Z}_G(X) \cup {R}_G(X)$ is a subgroup of $G$ in which ${Z}_G(X)$ has index $1$ or $2$.   
The group ${E}_G(X)$ is an extension of ${Z}_G(X)$ of degree at most two. 
In the group theoretical  set-up we refer to  {\cite[\S 2.1.4]{OS}},  \cite{BR}    on  reversing symmetries for groups.   

To find the reversing symmetric group ${E}_G(X)$, it is enough to construct one reversing element which is not in the centralizer.  
We have explicitly constructed an element in $ R_G(X) $ for each adjoint real semisimple element. Recall that for  a simply connected complex semisimple Lie group $ G $, the centralizer $ Z_G(X)$ of a semisimple element $ X$ is connected; see \cite[Theorem 2.3.3, p.28]{CoMc}. 
Thus the centralizer is determined by its Lie algebra which has a nice description in terms of Cartan subalgebra and certain root vectors; see \cite[Lemma 2.1.2, p. 20]{CoMc}.
Therefore, our construction also classifies the reverser of adjoint real  semisimple elements  in simple Lie groups.

\section{Notation and background}\label{notation} 
The Lie groups will be denoted by the capital 
letters, while the Lie algebra of a Lie group will be denoted by the corresponding lower case German letter. 
For a subgroup $H$ of $G$ and a subset $S$ of $\g$, the
subgroup $Z_{H} (S)$ of $H$ that fixes $S$ pointwise under the adjoint action is called the {\it centralizer}
of $S$ in $H$. Similarly, for a  Lie subalgebra $\h \,\subset\, \g$ and a subset $S \,\subset\, \g$, by $\z_\h (S)$ we will denote
the subalgebra of $\h$ consisting of all the elements that commute with every element of $S$. 
For $A\in {\rm M}_n(\C)$, $A^t$ denotes the transpose of the matrix $A$.
Let $ {\rm I}_n $ denote the $ n \times n$ identity matrix, and 
\begin{equation}\label{defn-I-pq-J-n}
	{\rm J}_n \,:=\, 
	\begin{pmatrix}
		& -{\rm I}_n  \\
		{\rm I}_n & 
	\end{pmatrix}\,.
\end{equation}
Here we will work with the following classical simple Lie groups and Lie algebras over $\C$:
\begin{align*}
{\rm SL}_n (\C)&:= \{g \,\in\, {\rm GL}_n (\C) \,\mid\, \det (g) \,=\,1  \},  &{\s\l}_n (\C):=\{z \,\in\, {\rm M}_n (\C) \,\mid\, \text{tr} (z) = 0 \};
	\\
{\rm SO} (n,\C)&:= \{g \,\in\, {\rm SL}_{n}(\C)\,\mid\, g^t g \,=\, {\rm I}_{n} \}, &{\s\o} (n,\C):= \{z \,\in\, \s\l_{n}(\C) \,\mid\, z^t {\rm I}_{n} + {\rm I}_{n} z =0 \};\\
{\rm Sp} (n,\C)&:= \{g \,\in\, {\rm SL}_{2n}(\C) \,\mid\, g^t {\rm J}_{n} g \,={\rm J}_{n} \}, &{\s\p} (n,\C):= \{z \,\in\, \s\l_{2n}(\C) \,\mid\, z^t {\rm J}_{n} + {\rm J}_{n} z =0 \}.
\end{align*}

For any group $H$, let $H^n_\Delta$ denote the diagonally embedded copy of $H$ 
in the $n$-fold direct product $H^n$. Similarly, for a matrix $A\in {\rm M}_n(\C)$, let $A^n_\Delta$ denote the diagonally embedded copy of $A$  in the $n$-fold direct sum $A\oplus \dots \oplus A$.  
For a pair of disjoint ordered sets $(v_1,\, \dots ,\,v_n)$ and $(w_1,\, \dots ,\,w_m)$, the ordered set $(v_1,\, \dots ,\,v_n,\, w_1,\, \dots ,\, w_m)$ will be denoted by
$$(v_1, \,\dots ,\,v_n) \vee (w_1,\, \dots ,\,w_m)\, .$$

For a  Lie algebra $ \g $ over $\C$, a subset $\{X,H,Y\} \,\subset\, \g$ is said to be a {\it $\s\l_2$-triple}  if $X \,\neq\, 0$, $[H,\, X] = 2X$, $[H,\, Y] =  -2Y$ and $[X,\, Y] =H$. 
Note that  for a $\s\l_2$-triple $\{X,H,Y\}$ in $\g$, $\text{Span}_\C \{X,H,Y\}$ is isomorphic to  $\s\l_2(\C)$. We now recall a well-known result due to Jacobson and Morozov.

\begin{theorem}[{Jacobson-Morozov, cf.~\cite[Theorem~9.2.1]{CoMc}}]\label{Jacobson-Morozov-alg}
	Let $X\,\in\, \g$ be a non-zero nilpotent element in a  semisimple Lie algebra $\g$ over $\C$. Then there exist $H,\,Y\, \in\, \g$ such that $\{X,H,Y\}$ is a $\s\l_2$-triple.
\end{theorem}

\section{Adjoint reality for  semisimple elements }\label{sec-reality-ss-elts-complex-g}
Let $G$ be  a complex simple Lie  group with Lie   algebra $\g$. Let $H\in \g$ be a semisimple element, and $ \h $ be a Cartan subalgebra in $ \g $.  Further we may assume $H\in \h$ as $ {\rm Ad}(g)(H)\in \h $ for some $ g\in G $.

\subsection{Semi-simple elements in $ \s\l_n(\C) $}\label{sec-semisimple-elt-sl-n-c}
Let $ \g:=\s\l_n(\C) $ and $ \h$ be the subalgebra consisting of all diagonal matrices in $ \g $. Then $ \h $ is  a Cartan subalgebra in $ \g $.  

\begin{lemma}
Let  $ H$ be a semisimple element in  $ \g\l_n(\C)$. Then $ H$ is   ${\rm Ad}_{{\rm GL }_n(\C)}$-real  in  $ \g\l_n(\C)$ if and only if  {        whenever $ \lambda  $ is an eigenvalue of $ H $,   $ -\lambda  $ is also an eigenvalue of $ H $ with the same multiplicity}.
\end{lemma}

\begin{remark}\label{semisimple-real-gln}
A semisimple element $ H \in  \g\l_n(\C) $  is ${\rm Ad}_{{\rm GL }_n(\C)}$-real  if and only if  $ H $  is  ${\rm Ad}_{{\rm SL }_n(\C)}$-real in $  \s\l_n(\C) $. The analogous statement is also true in the group  {      theoretic} sense; see \cite[p. 77]{OS}.
\end{remark}

\begin{theorem}
Every real semisimple element in $ {\rm Lie(\,PSL }_n(\C) )$ is strongly ${\rm Ad}_{{\rm PSL }_n(\C)}$-real.
\end{theorem}
\begin{proof}
Let $ H\in  {\rm Lie(\,PSL }_n(\C) )$ be a ${\rm Ad}_{{\rm PSL }_n(\C)}$-real semisimple element.  We may assume $ H= {\rm diag}(h_1, \ldots,h_m, -h_1,\ldots,-h_m,0,\ldots,0) $.  Then $ \sigma: ={\rm diag}( {\rm J}_m, \,\sqrt{-1}\,{\rm I}_s) $   will conjugate $ H $ and $ -H $, where $ 2m+s=n $,  and $ {\rm J}_n$  is as in  \eqref{defn-I-pq-J-n}. 
\end{proof}

\begin{example}\label{exp-diag-x-x}
Consider the semi-simple element $ H={\rm diag}(x_1, - x_1)$. Let $ g\in {\rm GL}_2(\C) $ so that $ gH = -Hg $. Then $ g$  is of the form $\left(\begin{smallmatrix}
0&b\\
c&0
\end{smallmatrix}\right) $. Hence $ H $ is  a strongly ${\rm Ad}_{{\rm GL}_2(\C)}$-real element  in $\g\l_{2}(\C)$,         and is an  ${\rm Ad}_{{\rm SL}_2(\C)}$-real element  but not strongly    ${\rm Ad}_{{\rm SL}_2(\C)}$-real in $\s\l_{2}(\C)$.   Note that $ g\in {\rm Sp}(1,\C) $ if $ bc=-1 $. Similarly,  $ H $ is  a  ${\rm Ad}_{{\rm Sp}(1,\C)}$-real element  but not strongly ${\rm Ad}_{{\rm Sp}(1,\C)}$-real        element  in $\s\p(1,\C)$.  \qed
\end{example}

The next  result  classifies  strong ${\rm Ad}_{{\rm SL}_n(\C)}$-reality in the Lie algebra $ \s\l_{n}(\C) $; see  \cite[Proposition 2.5]{GLM3}.  Here we provide a detailed proof.

 \begin{theorem}
 		\label{Pro-SS-sl-n-c}
An  ${\rm Ad}_{{\rm SL}_n(\C)}$-real semisimple element in $ \s\l_n(\C) $ is strongly ${\rm Ad}_{{\rm SL}_n(\C)}$-real if and only if either $ 0 $ is an eigenvalue or $ n\not \equiv 2 \pmod 4$.
 \end{theorem}
\begin{proof}
Let $ H\in  \s\l_n(\C)$ be a ${\rm Ad}_{{\rm SL}_n(\C)}$-real semisimple element. If $ 0 $ is an eigenvalue of $ H $, then using Example \ref{exp-diag-x-x} it follows that $ H $ is strongly ${\rm Ad}_{{\rm SL}_n(\C)}$-real. Suppose $ 0 $ is not an eigenvalue of $ H $, and  $ n\not \equiv 2 \pmod 4$, then   $ n=4m $ for $ m\in \N $ and we can assume $ H= {\rm diag}(h_1, \ldots,h_{2m}, -h_1,\ldots,-h_{2m}) $. Then, $ H $ and $ -H $ will be conjugated by $ g=\left( \begin{smallmatrix}
&{\rm I}_{2m}\\
{\rm I}_{2m}&\\
\end{smallmatrix}\right) $. 

Next assume that $ H $ is strongly ${\rm Ad}_{{\rm SL}_n(\C)}$-real and $ 0 $ is not an eigenvalue of $ H $. We will show $ n\in 4\N$. Without loss of generality, we can assume $ H= {\rm diag}(h_1, \ldots,h_m, -h_1,\ldots,-h_m) $ and $ gH=-Hg $ for some involution. 	Let $ e_j $ be the standard column vector in $ \C^n $ with $ 1 $ in $ j^{\rm th} $ place and $ 0 $ elsewhere. For $ 1\leq j\leq m $,
$$  
Hge_j\,=\, -gHe_j\,=\, -gh_je_j \,=\, -h_jge_j.
$$
Let $ V_j:=\C e_j \oplus \C ge_{j} $ and $\CC_j :=\{e_j, g e_{j} \}$. Since $ g^2 = {\rm I}_2 $, $ g(V_j) \subset V_j$. Then $ \{e_j, ge_j\mid 1\leq j\leq m \} $
 forms a basis of $ \C^{2m} $. 
 Set $\CC:=   \CC_1\vee\cds \vee  \CC_m$.  Then the matrix $[g]_ \CC $ is a $ 2\times2 $ block-diagonal matrix and $ \det \,[g]_ \CC = (-1)^m $.  As $ \det g =1 $, it follows that $ m\in 2\N $.  This completes the proof.
\end{proof}

\subsection{Semi-simple elements in $ \o(n,\C) $ and $\s\o(n,\C)$}\label{sec-semisimple-elt-so-n-c}
Up to conjugacy, any semisimple element in $ \o(n,\C) $ or $ \s\o(n,\C) $ belongs to the following Cartan  subalgebra $ \h $, see \cite[p. 127]{K}:
\begin{align}\label{semisimple-elt-so-n-c}
\h:= \begin{cases}
{\rm diag}(H_1, \ldots, H_m,0 ) & {\rm if } \, n=2m+1\\\
{\rm diag}(H_1, \ldots, H_m) & {\rm if } \, n=2m
\end{cases}, {\rm \, where \, }  H_j = \begin{pmatrix}
0 & x_j\\
-x_j &0
\end{pmatrix}, x_j\in \C.
\end{align}

\begin{example}\label{exp-so2c}
Consider $ H=\left( \begin{smallmatrix}
	0  & x\\
	-x &0
\end{smallmatrix} \right)\in \s\o(2,\C)=\o(2,\C) $, where $ x\in \C $. Let $ g \in {\rm GL}_2(\C) $ so that  $ gHg^{-1} =-H$. Then $ g$ is of the form  $ \left(\begin{smallmatrix}
	a & b\\
	b& -a
	\end{smallmatrix}  \right)$.
Hence,  $\det g = -a^2-b^2$, and   $gg^t= g^2=(-\det g){\rm I}_2 $. Thus, one can choose  $ g\in {\rm O}(2,\C) $ with $ \det g =-1 $. This shows that $ H $ is  a   {       ${\rm Ad}_{{\rm O}(n,\C) }$-real, as well as strongly ${\rm Ad}_{{\rm O}(n,\C) }$-real, }  element in $ {\o}(2,\C)$ but not   {       ${\rm Ad}_{{\rm SO}(n,\C) }$-real }  in $ {\s\o}(2,\C)$.       \qed
\end{example} 

The following result follows  from the construction done in the above example.
\begin{lemma}
	Every semisimple element in $ \o(n, \C) $ is  ${\rm Ad}_{{\rm  O}(n,\C) }$-real.
\end{lemma}

Next we  investigate strongly ${\rm Ad}_{{\rm  O}(n,\C) }$-real elements in $\o(n,\C)$.

\begin{proposition}\label{Onc}
Every semisimple element in $ \o(n, \C) $ is strongly ${\rm Ad}_{{\rm O}(n,\C) }$-real.
\end{proposition}	 

\begin{proof}
Enough to consider  the elements of $\h $. For the elements in $ \h $ one can easily construct   {      the }  required involution  using the Example \ref{exp-so2c}.
\end{proof}

Now we classify the strongly Ad$_{{\rm SO}(n,\C)}$-real semisimple element in  $\s\o(n,\C) $. 

\begin{theorem}\label{thm-sonC-semisimple} 
Let $ H\in \s\o(n,\C) $ be a  semisimple element. Then $ H $
 is a strongly Ad$_{{\rm SO}(n,\C)}$-real if and only if either $ 0$ is an eigenvalue of $ H $ or  $ n \not\equiv 2 \pmod 4$.   
\end{theorem}

\begin{proof}
Since $H$ is semisimple, there exists $ \sigma\in {\rm SO}(n,\C) $ such that $ {\rm Ad}(\sigma)H = {\rm diag}(H_1,\dots, H_m, \underset{r\text{-many}}{\underbrace{0,\dots , 0}}) $, $ r\geq 1 $ where $ H_j $ is as in \eqref{semisimple-elt-so-n-c}. 

First assume that $0$ is an eigenvalue of $H$. 
Let $ g:= {\rm diag} (\underset{m\text{-many}}{\underbrace{ {\rm I}_{1,1} ,\dots , {\rm I}_{1,1} }}, {\rm I}_r ) $, where $ {\rm I}_{1,1}:= {\rm diag}(1,-1) $. 
Then ${\rm Ad}(\sigma^{-1} g \sigma)H =- H$. Note that  $ g^2 = {\rm I}_n$. If $ \det g =1$, then we are done. Otherwise replace $ g $ by $  {\rm diag} ( {\rm I}_{1,1} ,\dots , {\rm I}_{1,1} , {\rm I}_{r-1}, -1 ) $ to get {  the}  required involution in $ {\rm SO}(n,\C) $. Next assume that $ 0 $ is not an eigenvalue, and $ n\not\equiv 2 \pmod 4 $. Thus $n \equiv 0 \pmod 4 $, and hence up to conjugacy $ H $ is of the form   ${\rm diag}(H_1,\dots, H_{2k}) $, $ 4k=n $. By choosing an involution $g $ as above with $ m=2k, r=0 $, we have that $ H $ is strongly ${\rm Ad}_{{\rm SO}(n,\C)}$-real.

Next suppose that $H$ is strongly Ad$_{{\rm SO}(n,\C)}$-real. Further assume that $ 0$ is not an eigenvalue of $ H $, and $ H=   {\rm diag}(H_1,\dots, H_m) $  where $ H_j $ is as in \eqref{semisimple-elt-so-n-c}.  Let $ e_j $ be the standard column vector in $ \C^n$. Then $e_{2j-1}+\sqrt{-1} e_{2j} $ and  $g(e_{2j-1}+\sqrt{-1} e_{2j})$ are eigenvector of $ H $ corresponding to the eigenvalues $\sqrt{-1} x_j$ and $-\sqrt{-1} x_j$, respectively. 
Let $ V_j:= \C(e_{2j-1}+\sqrt{-1} e_{2j} )\oplus \C g(e_{2j-1}+\sqrt{-1} e_{2j}) $, and $ \CC_j:= \{ e_{2j-1}+\sqrt{-1} e_{2j}, \, g(e_{2j-1}+\sqrt{-1} e_{2j})\} $. Since $ g^2= {\rm I_2} $,  $ g(V_j ) \subset V_j$. Then $ \{ e_{2j-1}+\sqrt{-1} e_{2j}, \, g(e_{2j-1}+\sqrt{-1} e_{2j}) \, \mid \, 1\leq j\leq m \} $ forms a basis of $ \C^{n} $. Set $\CC:=  \CC_1\vee \dots \vee  \CC_m$.   Then the matrix $[g]_ \CC $ is a $ 2\times2 $ block-diagonal matrix and $ \det \,[g]_ \CC = (-1)^m $.  As $ \det g =1 $, it follows that $ m\in 2\N $.  This completes the proof.
\end{proof}

\subsection{Semi-simple elements in $ \s\p(n,\C) $}\label{sec-semisimple-elt-sp-n-c}
Recall that in Example \ref{exp-diag-x-x},  $ {\rm diag}(x,-x) $ is not strongly Ad$_{ {\rm Sp}(1,\C) }$-real in    but we can choose a conjugating involution $ g $ from $ {\rm PSp}(1,\C) $.
Thus, we will first consider the semisimple element in the Lie algebra of $ {\rm PSp}(n,\C) $.  Let  $ \h:= \{ {\rm diag } (h_1,\ldots, h_m, -h_1, \ldots, -h_m)  \mid h_j \in \C \} $. 
\begin{theorem}
Every semisimple element in ${\rm Lie\,( PSp} (n, \C) )$ is strongly {\rm Ad}$ _{{\rm PSp}(n,\C) } $-real. 
\end{theorem}
\begin{proof}
Let $ X\in {\rm Lie\,( PSp} (n, \C) )$. Then  Ad$(g) X\in \h $ for some $ g\in {\rm PSp}(n,\C) $.  Thus we  define required  involution using Example \ref{exp-diag-x-x}. 
\end{proof}

The following  corollary is immediate.   
\begin{corollary}\label{cor-ss-spn}
Every semisimple element in $ \s\p(n, \C) $ is    {\rm Ad}$ _{{\rm Sp}(n,\C) } $-real. 
\end{corollary}

\begin{theorem}\label{thm-sp-nC}
A semisimple element in $\s\p(n,\C) $ is strongly  {\rm Ad}$ _{{\rm Sp}(n,\C) } $-real if and only if the  multiplicity of each non-zero eigenvalue  is even. 
\end{theorem}
\begin{proof}
Let $ H $ be a strongly real  semisimple element in $ \s\p(n,\C) $. We can assume  $ H ={\rm diag } (h_1,\ldots, h_n, -h_1, \ldots, -h_n)$, and $ gH=-Hg $ for some involution $ g\in {\rm Sp}(n,\C) $. 
Suppose $ h_j \neq 0$ and multiplicity of $ h_j $ is $ m $. Let $ \CC_j:= \{e_{j_1},\ldots, e_{j_{m}} \} $ be an ordered basis of the  eigenspace of $ H $ corresponding to the eigenvalue $ h_j $. Then  $ \CC_{n+j}:= \{e_{n+j_1},\ldots, e_{n+j_{m}} \} $ is an ordered basis of the  eigenspace corresponding to the eigenvalue $ -h_j $. Let $ V_j$ be the $ \C$-span of $  \CC_j\vee  \CC_{n+j} $. Then the involution  $ g$  keeps $V_j $ invariant and $ (g|_{V_j})^t{\rm J}_m (g|_{V_j})\,={\rm J}_m $, where $ {\rm J}_m $ is as  in \eqref{defn-I-pq-J-n}. As $ H(ge_{j_l})= -h_j ge_{j_l} $ for $ 1\leq l \leq m $, we can write
\begin{align}\label{g-on-eigen-sp-Vj}
 [g]_{ \CC_j\vee  \CC_{n+j} } := \big( \begin{smallmatrix}
0 &  B\\
  C&0
 \end{smallmatrix} \big)\,,\quad {\rm for \, some \,} B,C\in {\rm GL}_m(\C) \,.
\end{align}
Since $ g|_{V_j} $ is an involution and $ (g|_{V_j})^t{\rm J}_m (g|_{V_j})\,={\rm J}_m $, it follows that $C^{-1}= B= - B^t $. Hence, $ m $ (the multiplicity of $ h_j $) has to be even.
	
Next assume that  multiplicity of each non-zero eigenvalue  is even. Can assume any semisimple element $ H ={\rm diag } (h_1,\ldots, h_n, -h_1, \ldots, -h_n)$, where $ h_{2j-1}=h_{2j} $ for $ j=1,\ldots,n/2 $. In this case to  define a required involution $ g $, set $ B:= \begin{pmatrix}
&&{\rm J}_1\\
&\reflectbox{$ \ddots $}\\
{\rm J}_1
\end{pmatrix} $ in \eqref{g-on-eigen-sp-Vj}. This completes the proof.
\end{proof}

\section{${\rm Ad}_{{\rm Sp}(n, \C)}$-reality in $\s\p(n,\C)$} 
The aim of this section is to prove that every element in $\s\p(n,\C)$ is adjoint real. For this, we need to recall some known construction and results, cf. \cite{BCM}. The structure of the centralizer of nilpotent elements 
play an important role here.

 Let $0 \neq X\in \s\p(n,\C)$ be a nilpotent element.  Let   $\a$ be a $\s\l_2$-triple in $\s\p(n,\C)$ containing $X$; see Theorem \ref{Jacobson-Morozov-alg}. 
Note that  $\C^{2n}$  is $\C$-module over $\a$.  By decomposing $\C^{2n}$ as direct sum of irreducible $\a$-module, let $\N_{\d}:=\{d_1,\dots,d_s \}$ be the dimensions of  irreducible  $\a$-modules.  Let $M(d-1)$ denote the sum of all  $\C$-subspaces  of $\C^{2n}$ which are irreducible $\a$-submodule of dimension $d$.
Then $M(d-1)$ is the {\it isotypical component} of $\C^{2n}$ containing all the irreducible submodules
of $\C^{2n}$ with highest weight $d-1$.  Let
\begin{equation*}\label{definition-L-d-1}
	L(d-1)\,:= \,  \text{Span}_\C\{ v\in  M(d-1)\, \mid \,   \, {\rm weight \, of }\, v \, {\rm is} \, 1-d\,\}\, .
\end{equation*}
Then it follows that    $M(d-1)\,=\, L(d-1)\oplus XL(d-1)\oplus \dots \oplus X^{d-1}L(d-1)$; see \cite[Lemma A.1]{BCM}.  Let $t_{d} \,:=\, \dim_\C L(d-1).$    Then  $\sum_{d\in \N_{\d}} d t_{d}\,=\, 2n$, and   let $\{X^lv^{d}_j\, \mid \,  0\leq l < d, \, 1\leq j \leq t_{d}, d\in \N_{\d} \}$ be a basis of  $\C^{2n}$ as constructed in \cite[Lemma A.6]{BCM}.   We need to fix an ordering of  the above basis. 
Let $(v^d_1,\, \dotsc ,\, v^d_{t_d})$
be an ordered $\C$-basis of $ L(d-1) $  for $d \,\in\, \N_\d$. Then it follows that
\begin{equation*}\label{old-ordered-basis-part}
	\BC^l (d) \,:=\, (X^l v^d_1,\, \dotsc ,\,X^l v^d_{t_d})
\end{equation*}
is an ordered $\C$-basis of $X^l L(d-1)$ for $0\,\leq\, l \,\leq\, d-1$ with $d\,\in\, \N_\d$. 
Define
\begin{equation*}\label{old-ordered-basis}
	\BC(d) \,:=\, \BC^0 (d) \vee \cdots \vee \BC^{d-1} (d) \ \forall\ d \,\in\, \N_\d\, ,\ \text{ and }\ \BC\,:=\,
	\BC(d_1) \vee \cdots \vee \BC(d_s)\, .
\end{equation*}
Let $\langle \cdot,\, \cdot \rangle \,:\, \C^{2n} \times \C^{2n} \,\longrightarrow\, \C \ $  be the symplectic form  given by $\langle x,\, y \rangle \,: = x^t{\rm J}_n y$. 
Define another  form on $L(d-1)$ below, as in \cite[p.~139]{CoMc},
\begin{equation}\label{new-form}
	(\cdot  ,\,\cdot)_{d} \,:\, L(d-1) \times L(d-1 )\,\longrightarrow\, \C\quad  ;\ \ \
	(v,\, u)_d \,:=\,  \langle v \,,\, X^{d-1} u \rangle\,.
\end{equation}

The following result is due to Springer-Steinberg, which describes the structure of the centralizer of the $\s\l_2$-triple in ${\rm Sp}(n, \C)$ and $\s\p(n,\C)$. 

\begin{lemma}[{cf.~\cite[Theorem 6.1.3]{CoMc}, \cite[Lemma 4.4]{BCM}}]
	\label{reductive-part-comp} 
The following isomorphisms hold: 
	\begin{enumerate}
\item If $\a$ is a $\s\l_2$-triple in $\s\p(n, \C)$, then 
$$ 
\ZC_{{\rm Sp}(n, \C)} (\a) 
		=  \left\{  g \in {\rm SL}(\C^{2n})   \middle\vert    \begin{array}{cccc}
			g (X^l L(d-1))\, \subset \,  X^l L(d-1) ,\vspace{.14cm} \\
			\!  \big[g |_{ X^l L(d-1)}\big]_{{\BC}^l (d)} =   \big[g |_{  L(d-1)} \big]_{{\BC}^0 (d)} , (gx,gy)_d = ( x,y)_d   \vspace{.14cm}\\
			\text{ for all }  d\in \N_\d,~ 0 \leq l \leq d-1, \text{ and }  x,y \in L(d-1)
		\end{array} 
		\right\};
$$
		here $(\cdot,\, \cdot)_d$ is as in \eqref{new-form}.
		
\item   \label{reductive-part-comp-2}
In particular, 
\begin{align}\label{no-iso-centr}
\ZC_{{\rm Sp}(n, \C)} (\a)  \, &\simeq\, \Big\{ g \in \prod_{d \,\in\, \N_\d} {\rm GL} (L(d-1))	\,\bigm|   \,  \,       (gx,gy)_d = ( x,y)_d  ,\,                   \bigchi_\d (g) \,=\,1 \Big\}   \nonumber\\
&	\simeq\, \Big\{ \prod_{d \in \O_\d}{\rm  Sp}(t_\d/2,\, \C)_\Delta^d \, \times \, 	\prod_{d \in \E_\d}{\rm  O}(t_\d,\, \C)_\Delta^d		 \Big\}\,.
\end{align}

\item \label{reductive-part-comp-3}
	$$
\z_{{\s\p}(n, \C)} (\a) 
=  \left\{ A\in {\s\l}(\C^{2n})   \middle\vert    \begin{array}{cccc}
\!	A (X^l L(d-1))\, \subset \,  X^l L(d-1) ,\vspace{.14cm} \\
\!\!	\big[A |_{ X^l L(d-1)}\big]_{{\BC}^l (d)} =   \big[A |_{  L(d-1)} \big]_{{\BC}^0 (d)} , (Ax,y)_d + ( x,Ay)_d=0 \!\!\vspace{.14cm}\\
	\text{ for all }  d\in \N_\d,~ 0 \leq l \leq d-1, \text{ and }  x,y \in L(d-1)
\end{array} 
\right\}\!.
$$

\item
In particular, 
$$ \z_{{\s\p}(n, \C)} (\a)   \,\simeq\, \Big\{ 	 \Big(  \bigoplus_{d \in \O_\d}{\s\p}(t_\d/2,\, \C)_\Delta^d \Big)\,\,\, {{\bigoplus}} \, \,	\Big(\bigoplus_{d \in \E_\d}{\o}(t_\d,\, \C)_\Delta^d	\Big)	 \Big\}\,.$$
	\end{enumerate}
\end{lemma}

Now we will characterize the   ${\rm Ad}_{{\rm Sp}(n, \C)}$-real elements in $\s\p(n,\C)$.

\begin{theorem}\label{spnc-real}
Every element of  $\s\p(n,\C)$  is  ${\rm Ad}_{{\rm Sp}(n, \C)}$-real.
\end{theorem}

\begin{proof}
Let $X\in \s\p(n,\C)$, and $X=X_s\,+\, X_n$ be the Jordan decomposition of $X$ where $X_s$ and $X_n$ are the semisimple and nilpotent part of $X$, respectively. In view of  Corollary \ref{cor-ss-spn} and  \cite[Lemma 4.2]{GM},  we may further assume that $X_s\neq 0,\, X_n\neq 0$.  
  We will construct below  two  elements  $\sigma $ and $\tau$ in ${\rm Sp}(n, \C)$  so that
\begin{itemize}
	\item $\sigma \x \sigma^{-1} = -\x$ and $\sigma X_s \sigma^{-1} = X_s$,
	\item $ \tau X_s   \tau^{-1} = -X_s $ and $ \tau \x  \tau^{-1} = \x$. 
\end{itemize} 
 Then $\sigma\tau $ will do the job, i.e., $\sigma \tau X (\sigma \tau )^{-1}=-X $.
  
\noindent  {\it Construction of $\sigma$.}    Define $ \sigma \in {\rm GL}\,(\C^{2n} )$ below, as done in \cite[(4.3)]{GM} :
 \begin{align}\label{real-nilpotent-sp-n-c}
 	\sigma (X^l v^{d}_j ) := \begin{cases}
 		(-1)^l X^{l} v^{d}_j &  \qquad   \text{\ if\ } d\in \O_\d   \,,   \\
 		(-1)^{l}\sqrt{-1} X^l v^{d}_{j} &\qquad \text{\ if\ $ d\in \E_\d$}\, . 
 	\end{cases}		
 \end{align}
 Note that $ \sigma \x=-\x \sigma  $, and $\langle \sigma x, \sigma y \rangle = \langle x,y \rangle$ for all $ x,y\in \C^{2n} $. This shows that $ \sigma \in {\rm Sp}(n,\C)$; cf. \cite[Section 4.3]{GM}.   
 
The element $X_s$ is a semisimple element and commutes with $\x$, and $ \z_{{\s\p}(n, \C)} (\a)   $ is reductive part (Levi part) of $\z_{{\s\p}(n, \C)} (\x) $; see \cite[Lemma 3.7.3]{CoMc}.
Thus, $ X_s \in   \z_{{\s\p}(n, \C)} (\a)   $.   Write $[X_s|_{L(d-1)}]_{\BC^0_d} := X_{sd}$. Then using Lemma \ref{reductive-part-comp}\eqref{reductive-part-comp-3},  $X_s = \oplus_{d\in \N_\d} (X_{sd})_\Delta ^d$.
 Also $[\sigma]_{\BC^0(d)}$ is a scalar matrix; see $\eqref{real-nilpotent-sp-n-c}$.  
 Thus from Lemma \ref{reductive-part-comp}(3), it follows that  $\sigma X_s \sigma^{-1} = X_s$.

\noindent  {\it Construction of $\tau$.}   Note that 
  $$
  X_{sd}\,\in \, 
  \begin{cases}
{\s\p}(t_d/2,\, \C)    &   \   {\rm for} \  \th\in \O_{\d}\\  	
{\o}(t_d,\, \C)    &   \   {\rm for}\  \th\in \E_{\d}\,.
  \end{cases}
  $$
In view of Corollary \ref{cor-ss-spn} and Proposition \ref{Onc}, there exists
 $\tau_{d}\in 
\begin{cases}
	{\rm Sp}(t_d/2,\, \C)    &     {\rm for}\    \th\in \O_{\d}\\  	
	{\rm O}(t_d,\, \C)    &     {\rm for}\    \th\in \E_{\d}
\end{cases}$
such that  $\tau_d X_{sd}\tau_d^{-1}\,=\,-X_{sd}$  for all $d\in \N_{\d}$.   Finally set   $    \tau:=    \bigchi(\tau_{d})$, where $\bigchi$ is an isomorphism  in \eqref{no-iso-centr}; see \cite[Section 3.4]{BCM2} for such isomorphism. 
 Then   $ \tau X_s   \tau^{-1} = -X_s $ and $ \tau \x  \tau^{-1} = \x$. This completes the proof.
 \end{proof}

\subsection*{Acknowledgment} 
 Gongopadhyay acknowledges SERB  grant CRG/2022/003680.    The  authors thank  Tejbir for his comments.

\end{document}